\documentclass[11pt]{amsart}
\usepackage{amsmath, amsthm, amssymb}
\usepackage{verbatim}
\usepackage[colorlinks=true,citecolor=blue]{hyperref}

\renewcommand{\d}{\partial}
\newcommand{\dbar}{\overline{\partial}}
\newcommand{\ddbar}{\sqrt{-1}\d\overline{\d}}

\newcommand{\oep}{\omega_\ep}

\newtheorem{thm}{Theorem}
\newtheorem{prop}[thm]{Proposition}
\newtheorem{lem}[thm]{Lemma}

\newtheorem{conj}[thm]{Conjecture}

\theoremstyle{definition}

\newtheorem{remark}[thm]{Remark}

\renewcommand{\[}{\begin{equation}}
\renewcommand{\]}{\end{equation}}

\newcommand{\al}{\alpha}
\newcommand{\bal}{\bar\alpha}
\newcommand{\be}{\beta}
\newcommand{\bbe}{\bar\beta}
\newcommand{\ga}{\gamma}
\newcommand{\la}{\lambda}

\newcommand{\ep}{\epsilon}

\newcommand{\Ga}{\Gamma}

\newcommand{\vep}{\varepsilon}

\newcommand{\ZZ}{\mathbb{Z}}

\newcommand{\RR}{\mathbb{R}}
\newcommand{\CC}{\mathbb{C}}
\newcommand{\PP}{\mathbb{P}}

\newcommand{\s}{\mathbf{s}}
\newcommand{\sL}{\mathcal{L}}
\newcommand{\sD}{\mathcal{D}}
\newcommand{\bM}{Bl_pM}
\newcommand{\bC}{Bl_0\mathbb{C}}
\newcommand{\sO}{\mathcal{O}}
\newcommand{\bl}{\mathbf{l}}
\newcommand{\sF}{\mathcal{F}}
\newcommand{\sN}{\mathcal{N}}
\newcommand{\sU}{\mathcal{U}}

\title[Blowing up constant scalar curvature K\"ahler surfaces]{Expansions of solutions to extremal metric type equations on blow-ups of cscK surfaces}
\author[V. V. Datar]{Ved V. Datar}
\address{Department of Mathematics, UC Berkeley}
\email{vvdatar@berkeley.edu}
\thanks{Research supported by NSF RTG grant DMS-1344991}

\begin{document}

\begin{abstract} The aim of this article is to study expansions of solutions to an extremal metric type equation on the blow-up of constant scalar curvature K\"ahler surfaces. This is related to finding constant scalar curvature K\"ahler (cscK) metrics on $K$-stable blow-ups of extremal K\"ahler surfaces \cite{Sz12}. \end{abstract}

\maketitle

\section{Introduction}

A K\"ahler manifold $(M,\omega)$ is said to be {\em extremal} if the scalar curvature $\s_\omega$ is the Hamiltonian of a holomorphic vector field. Constant scalar curvature K\"ahler metrics are a particularly interesting subclass of extremal metrics. A basic problem in K\"ahler geometry is to study conditions which guarantee the existence of cscK or extremal metrics. The Yau-Tian-Donaldson conjecture \cite{Yau93, Ti97, Don02} relates existence of cscK metrics to an algebro-geometric notion of stability. Although this has been confirmed in the case of K\"ahler-Einstein metrics on Fano manifolds \cite{CDS1, CDS2, CDS3, Ti15}, the conjecture  in full generality is still wide open. In light of this, the problem of existence of extremal metrics on the blow-ups of extremal manfiolds has received considerable attention in recent years. 

The first major existence result was obtained by Arezzo and Pacard \cite{AP06}, who showed that if the automorphism group of an extremal manifold $(M,\omega)$ is discrete (in particular $\omega$ is cscK), then for any $p\in M$, the blow-up $\bM$ admits constant scalar curvature metrics in the K\"ahler class $L_\vep = \pi^*[\omega] - 2\pi\vep^2[E]$ for all small $\vep>0$. Here $E$ is the exceptional divisor and $\pi:\bM\rightarrow M$ is the blow-down map. In the presence of continuous automorphisms, sufficient conditions for existence of extremal metrics on $\bM$ were provided in \cite{AP09} and \cite{APS11}.

Making contact with the Yau-Tian-Donaldson conjecture, in \cite {Sz12,Sz15}, Sz\'ekelyhidi proved that if $m>2$, and $(\bM,L_\vep)$ is $K$-stable with respect to some special degenerations for all small $\vep$, then $\bM$ admits extremal metrics in the K\"ahler class $L_\vep$ for all sufficiently small $\vep$. The degenerations needed to test K-stability are precisely the ones generated by holomorphic vector fields on $M$. The key insight is to relate $K$-stability of the pair $(\bM,L_\vep)$ to finite dimensional GIT stability of the point $p$ with respect to the polarization $L + \delta K_M.$ This in turn is related to existence of extremal metrics via the following conjecture.
\begin{conj}\cite{Sz12}\label{Conj} Let $M$ admit a cscK metric $\omega$. 
\begin{itemize}
\item If $m>2$, there exist $\vep_0$ and $\delta_0$ such that if $p$ is GIT stable with respect to the polarization $L+\delta K_M$ for all $\delta<\delta_0$, then $(\bM_p,L_\vep)$ admits an extremal metric for all $\vep<\vep_0$.
\item If $m=2$, we can ask an analogous question with the polarization $L\pm \delta K_M$, where the sign is taken to be positive if $K_M\cdot L\geq 0$ and negative otherwise.
\end{itemize}
\end{conj}
While Sz\'ekelyhidi solved this conjecture in \cite{Sz15} for $m>2$, it is still open in dimension two. To describe his approach and to motivate our main theorem, we need to introduce some notation. Let $G$ be the compact group of holomorphic Hamiltonian isometries of $(M,\omega)$. The Lie algebra of $\mathfrak{g}$ corresponds to holomorphic Killing fields with zeroes \cite{LS94}. We normalize the moment map by $\mu:M\rightarrow \mathfrak{g}^*,$ so that $$\int_M\langle\mu,\xi\rangle\frac{\omega^m}{m!} = 0$$for all $\xi \in \mathfrak{g}$. We can then identify $\mathfrak{g}^*$ with $\mathfrak{g}$ by using the inner product $$\langle\xi,\eta\rangle = \int_M\langle\mu,\xi\rangle\langle\mu,\eta\rangle\frac{\omega^m}{m!},$$ and think of $\mu$ as a map from $M$ into $\mathfrak{g}$. Identifying elements in $\mathfrak{g}$ with their Hamiltonians, we can furthermore think of $\mu(p)$ as a Hamiltonian function on $M$ for each $p$. Let $T\subset G$ be a torus fixing $p$ with Lie algebra $\mathfrak{t}$. Then $T$ can be identified with a compact subgroup of the automorphism group of $\bM$, and so it makes sense to search for extremal metrics invariant under the action of $T$. 

Now fix a family of metrics $\oep\in L_\vep$ on $\bM$.  If $f$ is the Hamiltonian of a vector field in $\mathfrak{t}$, we denote the Hamiltonian (with respect to $\oep$) of the lifted vector field by $\bl(f)$.
Then the metric $\Omega = \oep + \ddbar u$ is a $T$-invariant extremal metric if and only if $$\s(\Omega) = \bl_{\Omega}(f) := \bl(f) + \mathrm{Re}(\nabla \bl(f)\cdot \nabla u),$$ for some $f\in \mathfrak{t}$. It turns out that it is easier to solve the more general equation \eqref{main thm equation} below, allowing $f$ to lie instead in the centralizer $\mathfrak{h}$ of $\mathfrak{t}$. The problem of existence of an extremal metric then reduces to checking the finite dimensional condition of whether $f$ lies in $\mathfrak{t}$.  To exploit the hypothesis in the conjecture above, the strategy in \cite{Sz15} was to obtain an expansion of $f$. It is shown in \cite{Sz15}, that if $m>2$, then one can find a solution with $$f_{p} = \s_\omega + (c_1-c_2\vep^2)\vep^{2m-2}\mu(p) + c_3\vep^{2m}\Delta\mu(p) + O(\vep^\kappa)$$ for some $\kappa>2m$, and constants $c_1,c_2$ and $c_3$. The crucial term that makes the entire argument work is the $\Delta\mu$ term. From general considerations, it is expected that for $m=2$, there will be no $\Delta\mu$ term of order $O(\vep^4)$. Our main result confirms this, and we obtain an expansion of $f$ up to the fourth order. As we will discuss below, even in obtaining the $O(\vep^4)$ term there are new difficulties in dimension two that are absent in higher dimensions. It is expected, at least when $K_M\cdot L\neq 0$, that there will be a $\Delta\mu$ term of order $O(\vep^6)$. A global Newton-iteration type argument, such as the one used in \cite{MMP06}, might be more useful in obtaining higher order terms in the expansion.

To state our main result, let the space of Hamiltonian functions (including constants) of vector fields in $\mathfrak{h}$ be denoted by $\overline{\mathfrak{h}}$ and also let $C^{\infty}(\bM)^T$ denote the space of all smooth $T$-invariant functions.

\begin{thm}\label{main}
Let $(M,\omega)$ be a constant scalar curvature K\"ahler surface, and let $T\subset G$ be a non-trivial torus. With $\oep$ and $L_\vep$ as above, there exists an $\vep_0$ depending only $M$, $\omega$ and $T$ such that the following holds. For any $\vep\in (0,\vep_0)$ and $p\in M$ fixed by $T$, there exists a $u\in C^{\infty}(\bM)^T$, and an $f\in \overline{\mathfrak{h}}$ such that 
\begin{equation}\label{main thm equation}
\s(\oep+\ddbar u) = \bl_{\oep+\ddbar u}(f),
\end{equation} where $f$ has the expansion $$f = \s_\omega -2\pi^2\vep^2(V^{-1}+\mu(p)) + \frac{\pi\s_\omega}{2}\vep^4(V^{-1}+\mu(p))+ O(\vep^\kappa),$$for some $\kappa>4.$  Here $V$ denotes the volume of $(M,\omega)$, and the constant in $O(\vep^\kappa)$  depends only on $M$, $\omega$ and $T$.
\end{thm}

 We refer the reader to section 2.3 for a definition of the lifting operator $\bl_\Omega$ with respect to a background metric $\Omega$ on $\bM$ and applied to a general $h\in \overline{\mathfrak{h}}$. 
 \begin{remark} If we instead assume that $\omega$ is only an extremal metric (with possibly non constant scalar curvature) and that $\nabla \s_\omega\in \mathfrak{t}$, the expansion should continue to hold with $\s_\omega$ replaced by $\s_\omega(p)$. 
 \end{remark}
 We end this section by outlining our method of proof.
\subsection{Overview of the proof} The general idea is that to obtain an expansion of $f$ correct to an order of $O(\vep^\kappa)$, one needs to construct an approximate solution that needs to be perturbed by only an $O(\vep^\kappa)$ amount to the actual solution. The first step is to construct the metrics $\oep$ by gluing $\omega$ to a rescaled copy of a scalar flat asymptotically Euclidean metric $\eta$ on $\bC^2$. The metric $\eta$ has a log-term which has to be cut-off when gluing to $\omega$. As in \cite{APS11, Sz12, Sz15}, a better approximate solution is constructed by introducing the Green's function $G$ of the Lichnerowicz operator with a log singularity at $p$. The new difficulty in dimension two is that the error incurred by introducing such a term is already of the order of $O(\vep^{2m})$, or $O(\vep^4)$. So we need to add an additional correction term. This is done by solving a linear problem on $M\setminus\{p\}$. The heart of the matter (Proposition \ref{residue} and Lemma \ref{modifying Green}) is then to check that the principle error term is orthogonal to the kernel of the adjoint, and so can be solved away. The rest of the estimates are then slight modifications of the estimates in \cite{Sz12} and \cite{Sz15}.  The non-triviality of $T$ only plays a role in adjusting the Green's function to have no linear terms locally. 
\subsection{Notations} $\Delta$ and $\nabla$ will denote respectively the {\em complex} Laplacian and the (1,0) part of the real gradient, with respect to either $\omega$ or $\oep$. All geometric quantities with respect to the Euclidean metric will be denoted by a subscript of $0$. For instance the Euclidean complex Laplacian is denoted by $\Delta_0$.
 
\section{Preliminaries}
\subsection{Extremal metrics} Extremal metrics on a K\"ahler manifold $X$ are critical points of the Calabi functional $$Ca(\Omega) = \int_{X}(\s(\Omega) - \hat\s(\Omega))^2$$ in a given K\"ahler class. Here $\hat\s(\Omega)$ denotes the average of the scalar curvature $\s(\Omega)$. The Euler-Lagrange equation is $$\bar\partial\nabla_\Omega\s(\Omega) = 0,$$ that is $\nabla_\Omega\s(\Omega)$ is a holomorphic vector field. A priori, the extremal metric equation is a sixth order equation for the K\"ahler potential. The following elementary observation (cf. \cite{APS11}) reduces this to a more manageable fourth order equation. Let $T$ be a non-trivial torus acting on $X$ with Lie algebra $\mathfrak{t}$.

 \begin{lem}
If there exists $\Phi\in C^{\infty}(X,\RR)^T$ and a $f\in \mathfrak{t}$ such that $$\s(\Omega+\ddbar \Phi) = f + \nabla_\Omega f\cdot\nabla_\Omega \Phi,$$ then $\Omega_\Phi :=\Omega+\ddbar\Phi$ is extremal.
\end{lem}
Note that since $\Phi$ is $T$-invariant, if we denote $\xi:= \nabla_\Omega f$, then $\xi(\Phi) = \bar\xi(\Phi)$, and so the right hand side in the above equation is real, as it should be.

For any K\"ahler metric $\Omega$, and an $\Omega$- pluri subharmonic function $\varphi$. we write $$\s(\Omega+\ddbar \varphi) = \s(\Omega) + \sL_\Omega(\varphi) + Q_\Omega(\varphi),$$ where $$\sL_\Omega(\varphi) = - \Delta_\Omega^2\varphi - R(\Omega)_{\al\bar \be}\nabla ^{\bar \be}\nabla^\al \varphi$$ is the linearization of the scalar curvature, and $Q_{\Omega}(\varphi)$ collects all the non-linear terms. We will denote the nonlinear term of order $k$ by a superscript. So for instance, 
\begin{align}\label{Q formula}
Q_\Omega^{(2)}(\varphi) &:= \frac{1}{2}\frac{d}{dt}\Big|_{t=0}\s(\Omega + t\ddbar\varphi) \nonumber\\
&= \varphi_{\al\bar\be}(\Delta\varphi)_{\be\bar \al} + \frac{1}{2}\Delta|\nabla\bar\nabla\varphi|^2 + R_{\al\bar\be}\varphi_{\lambda\bar \al}\varphi_{\be\bar\lambda},
\end{align}
where the Laplacian, covariant derivatives and Ricci curvature are with respect to the metric $\Omega$. We will also denote the nonlinear terms of order greater than $k$ by $Q_{\Omega}^{(\geq k)}(\varphi)$. The linearization is an elliptic operator, and it will be crucial to understand it's inverse. It is more convenient to  work instead with the self adjoint Lichnerowicz operator $$\sD^*_\Omega\sD_\Omega\varphi,$$ where $\sD_\Omega\varphi : = \dbar\nabla_\Omega\varphi,$ and $\sD^*_\Omega$ is the formal adjoint with respect to the inner product induced by $\Omega$ on the relevant bundle. A standard computation then shows that
\begin{equation}\label{eqn:lich}
\sL_\Omega(\varphi) = -\sD_\Omega^*\sD_\Omega(\varphi) + \nabla \varphi\cdot \nabla\s(\Omega),
\end{equation}
 and so the two operators coincide when $\Omega$ has constant scalar curvature.

 Now let us specialize to the case of a cscK K\"ahler surface $(M,\omega)$. Recall that $T\subset G$ is a torus fixing a point $p\in M$. We denote the centralizer of $T$ by $H\subset G$ with Lie algebra $\mathfrak{h}$. We also denote the corresponding Hamiltonians by $\overline{\mathfrak{h}}$. From definitions, it is clear that any $f\in\overline{\mathfrak{h}}$  is in the $T$-invariant kernel of the operator $\sD^*_\omega\sD_\omega$ and that any $h\in \overline{\mathfrak{h}}$ is orthogonal to the range of $\sD^*_\omega\sD_\omega$. The next Lemma is not required in what follows, but serves to  motivate Proposition \ref{residue}. Unlike the proof below, the proof of Proposition \ref{residue} requires a more involved integration by parts argument. 
 
  \begin{lem}\label{orthogonal}
For any smooth $\varphi$, and any $f\in \mathfrak{h}$, $$\int_M f \Big(Q^{(2)}_\omega(\varphi) + \mathrm{Re}(\nabla(\sD^*_\omega\sD_\omega \varphi) \cdot\nabla \varphi)\Big)\omega^2 = 0.$$
\end{lem}
\begin{proof}
If $\omega_t = \omega+t\ddbar\varphi$, then $f_t = f+t\nabla  f\cdot\nabla \varphi$ is the Hamiltonian for $\nabla f$ with respect to $\omega_t$. It follows that $$\int_M\sD^*_{\omega_t}\sD_{\omega_t}(\varphi) f_t~\omega_t^2 $$ is a zero for all $t$. Differentiating at $t=0$, we obtain the identity, 
\begin{equation*}
\int_M\frac{d}{dt}\Big|_{t=0}\sD_{\omega_t}^*\sD_{\omega_t}(\varphi) f~\omega^2  +\int_M\sD^*_\omega\sD_\omega(\varphi) \nabla f\cdot\nabla\varphi~\omega^2 + \int_M\sD^*_\omega\sD_\omega(\varphi)f\Delta\varphi~\omega^2=0. 
\end{equation*} 
On the other hand, differentiating \eqref{eqn:lich} for $\omega_t$, since $\s_\omega$ is a constant, it follows that $$\frac{d^2}{dt^2}\Big|_{t=0}\s(\omega_t) = -\frac{d}{dt}\Big|_{t=0}\sD_\omega^*\sD_\omega(\varphi) -  \nabla\varphi\cdot \nabla(\sD^*_\omega\sD_\omega \varphi),$$ and so $$\frac{d}{dt}\Big|_{t=0}\sD_{\omega_t}^*\sD_{\omega_t}(\varphi) = -2Q_\omega^{(2)}(\varphi) - \nabla\varphi\cdot\nabla(\sD^*_\omega\sD_\omega \varphi).$$ Substituting in the identity above, and integrating the last two terms by parts, we obtain the required result.
\end{proof}

\subsection{A K\"ahler metric on $\bM$} The aim of this section is to construct an approximatly constant scalar curvature metric on $\bM$ by gluing a scalar flat metric on $\bC^2$ to $\omega$. By rescaling the metric we can assume that there exist normal coordinates on the unit ball $B_1$ centered at $p\in M$. We begin with the following elementary observation, and refer the reader to a proof in \cite{ALM16}.
\begin{lem}\label{normal}
The metric $\omega$ has the form $$\omega = \ddbar\Big(\frac{|z|^2}{2} + \varphi_\omega\Big),$$ where $\varphi_\omega=O(|z|^4)$. Moreover, $$\varphi_\omega = \sum_{k=4}^\infty A_k(z),$$ where $A_k(z)$ is a degree $k$ homogenous real polynomial. In addition one has that 
\begin{align*}
\Delta_0^2A_4 &= -\s_\omega,\\
\Delta_0^2A_5 &= 0.
\end{align*}
\end{lem}
It is in fact possible \cite[Theorem 5, pg. 55]{BM48} to choose normal coordinates such that $T$ acts by unitary transformations. The action of $T$ then also lifts to an action on $\bC^2$, and allows us to work $T$-equivariantly.

We next describe a well known scalar flat K\"ahler metric on $\bC^2$. Recall that $\bC^2$ is the tautological bundle $\sO_{\PP^1}(-1)$. If $\tilde \pi:\bC^2\rightarrow \CC$ is the blow-down map, $p:\bC^2\rightarrow \PP^1$ is the bundle map, and $\omega_{FS}$ is the Fubini-Study metric normalized so that $\PP^1$ has volume $2\pi$, then $$\eta = \tilde \pi^*\omega_{Euc} + p^*\omega_{FS}$$ is a complete K\"ahler metric on $\bC^2$, called the {\em Burns-Simanca metric} (cf. \cite{Simanca91,Lebrun98}). In Euclidean coordinates on $\CC^2\setminus\{0\}$ we have the formula $$\eta = \ddbar\Big(\frac{|w|^2}{2} + \log{|w|}\Big).$$  From this it is clear that the metric is asymptotically locally Euclidean.

Next, let $\ga:[0,\infty)\rightarrow \mathbb{R}$ be a smooth function such that $\ga(t)\geq 0$, $\ga(t) = 0$ on $t<1$ and $\gamma(t) = 1$ on $t\geq 2$. For $\al\in(2/3,1)$, let $$r_\vep = \vep ^{\al},$$ and define cut-off functions $\ga_1(t) = \ga(|z|/r_\vep)$ and $\ga_2 = 1-\ga_1$. We then define a real closed $(1,1)$ form on $M_p$ by $\omega_\epsilon$ to be simply  $$\oep = \pi^*\omega + \vep^2\ddbar{\ga_2\log{|z|}}.$$ Note that $\oep = \omega$ on $M\setminus B_{2r_\vep}$ and is a small perturbation of $\vep^2\eta$ on $B_{r_\vep}$. 

\begin{lem}[cf. \cite{SySz16}]\label{zero approx}
For sufficiently small $\vep>0$, $\oep$ extends to a smooth K\"ahler metric on $\bM$ in the class $c_1(L_\vep).$  
\end{lem}

\begin{proof}The proof in our case is even simpler than the one in \cite{SySz16}, since the leading non-Euclidean term of $\omega$ is of order $O(|z|^4)$.
We only need to consider the domain $B_{2r_\vep}$. It is more convenient to work with the form $\vep^{-2}\oep$ on $\bC^2$. Changing coordinates $w = \vep^{-1}z$, the ball is transformed to $\tilde B_{2R_\vep}\subset \bC^2$ where $R_\vep = \vep^{-1} r_\vep = \vep^{\al -1}.$
\begin{itemize}
\item {\bf The region $|w|<1$}. On this region $$\vep^{-2}\oep = \eta + \ddbar\Psi_0$$ where $\Psi_0(w) = \vep^{-2}\varphi_4(\vep w).$ The region can be covered by open sets  $$U_1 = \{(w^1,w^2)~|~ |w^2|<1,~|w^1|>|w^2|/2\},~U_2 = \{(w^1,w^2)~|~ |w^2|<1,~|w^2|>|w^1|/2\}$$On $U_1$ we introduce the ``blow-up" coordinates, $$x = w^1/w^2,~ w^2 = y.$$ Recalling that $\bC^2 = \sO_{\PP^1}(-1)$, the $x$-coordinate then is a coordinate on the base $\PP^1$, and the $y$-coordinate represents the fiber direction. Clearly $|x|<2$ and $|y|<1$. We can symmetrically introduce coordinates on $U_2$.  In these coordinates, it is not difficult to see that 
\begin{equation} \label{estimate for psi}
||\nabla^k\Psi|| = O(\vep^2),
\end{equation} for all $k\geq 2$. This follows from the fact that the leading term in $\varphi_4$ is a fourth order homogenous polynomial in $(w^1,w^2)$ and hence also $(x,y)$. Since $\eta$ is a K\'ahler metric on this region, if $\vep>0$  is sufficiently small, $\vep^{-2}\oep$ extends to a K\"ahler metric.
\item{\bf The region $1<|w|<R_\vep$.} Here again the cut-off function is one, so like before we write $$\vep^{-2}\oep - \eta = \ddbar\Psi_0.$$ Now $$|\nabla^k \Psi_0| = O(\vep^2R_\vep^{4-k}) = O(\vep^{4\al -2}) = o(\vep^{2/3}),$$ and so once again for sufficiently small $\vep>0$ $\vep^{-2}\oep$ is a positive form. 

\item{\bf The region $R_\vep<|w|<2R_\vep$.}  Here we think of the metric as a perturbation of $\vep^{-2}\omega$. That is, $$\vep^{-2}\oep = \vep^{-2}\omega + \ddbar(\ga_2\log{|w|}).$$ But $$||\nabla^2\ga_2\log{|w|}|| = O(R_\vep^{-2}),$$ and so for small $\vep>0$ this is again a positive form. 

 \end{itemize}
 
Since the volume of the exceptional divisor is $2\pi$, and $\oep$ is equal to $\omega$ away from a small region containing the exceptional divisor, $\oep \in c_1(L_\vep).$\end{proof}

\subsection{Lifts of Hamiltonian functions} In this section we define the lifting operator $$\bl : \overline{\mathfrak{h}}\rightarrow C^{\infty}(\bM,\RR)^T$$ that shows up in the statement of the theorem. Writing $\overline{\mathfrak{h}} = \overline{\mathfrak{t}}\oplus \mathfrak{h}'$ such that $h(p) = 0$ for all $h\in \overline{\mathfrak{h}}'$ we define $$\bl(h) = \begin{cases}h + \vep^2\mathrm{Re}\langle\nabla h,\nabla \ga_2\log|z|\rangle,~h\in \overline{\mathfrak{t}}\\ h,~h\in \mathfrak{h}'. \end{cases}$$

If $h\in \overline{\mathfrak{t}}$, then $\xi = \nabla h$ vanishes at $p$, and hence lifts to a holomorphic vector field $\tilde \xi$ on $\bM$. As observed in \cite{SySz16}, if $h\in \overline{\mathfrak{t}}$, $\bl(h)$ defined as above extends to a smooth function on $\bM$, and is in fact precisely the Hamiltonian of $\tilde \xi$ with respect to the metric $\oep$. Note also, that in either case, $$\bl(h) = h$$ on $|z|>2r_\vep$. We will also need to define lifts for metrics other than $\oep$. So if $\Omega = \oep + \ddbar u$, we define $$\bl_{\Omega}(h) = \bl(h) + \mathrm{Re}\langle \nabla \bl(h), \nabla u\rangle,$$ where the gradient is taken with respect to $\oep$. 

\section{Analysis in weighted Holder spaces} In this section, we will study the mapping properties of the linear operator $\sF:C^{4,\al}_{loc}(\bM)\times \overline{\mathfrak{h}}$ defined by

$$\sF_{\oep}(\varphi,f) := \sL_{\oep}(\varphi) - \bl(f).$$
In particular, we will be concerned with uniform bounds for the inverse of this operator. The relevant function spaces are certain {\em weighted H\"older spaces}. We will now give a brief review of these spaces. For more details, the reader should refer to \cite{PacRiv00,Szbook}.

For any $f\in C^{k,\al}_{loc}(M_p)$ and $\delta\in \RR$, we define the weighted norm $$||f||_{C^{k,\al}_{\delta,\omega}}(M_p) = ||f||_{C^{k,\al}_\omega(M\setminus B_{1/2})} + \sup_{r\in (0,1/2)}r^{-\delta}||f||_{C^{k,\al}_{r^{-2}\omega}(B_{2r}\setminus B_r)},$$ where the $\omega$ and $r^{-2}\omega$ indicate the metric used to measure the norms. The weighted H\"older space is then a subspace of $ C^{k,\al}_{loc}(M_p)$ consisting of all functions with finite $||\cdot||_{C^{k,\al}_\delta}$ norm. It can be shown that this is a Banach space. 

The next Lemma is a consequence of the duality theory for weighted H\"older spaces in \cite{Bart, Mel93} (see also the discussion in \cite[Remark 4.5]{GV16}).  The second part follows from the fact that for $\delta>0$, any $f\in C^{4,\al}_{\delta}(M_p)$ in the kernel of $\sD_\omega^*\sD_\omega$ is in fact smooth, and hence is the Hamiltonian of a holomorphic vector field.
\begin{lem}\label{lin M_p} If $\delta\notin \ZZ$, the image of $$\sD_\omega^*\sD_\omega:C^{4,\al}_{\delta}(M_p)^T\rightarrow C^{0,\al}_{\delta-4}(M_p)^T$$ is the orthogonal complement of 
the kernel of $$\sD_\omega^*\sD_\omega:C^{4,\al}_{-\delta}(M_p)^T\rightarrow C^{0,\al}_{-\delta-4}(M_p)^T.$$ In particular, if $\delta\in (-1,0)$ and $v\in C^{0,\al}_{\delta-4}(M_p)^T$ such that for any $f\in \overline{\mathfrak{h}}$ with $f(p) = 0$, $$\int_{M_p}vf\omega^2 = 0,$$ then there exists a $u\in C^{4,\al}_{\delta}(M_p)^T$ such that $$\sD_\omega^*\sD_\omega u = v.$$ 

\end{lem}

Next, we define the weighted norm on $\bC^2$ with respect to $\eta$ by $$||f||_{C^{k,\al}_{\delta,\eta}} = ||f||_{C^{k,\al}_{\eta}(B_2)} + \sup_{R\geq 1}R^{-\delta}||f||_{C^{k,\al}_{R^{-2}\eta}(B_{2R}\setminus B_R)}.$$

The weighted H\"older space on $\bM$ is then defined by interpolating between the two norms defined above. More precisely, we set $$||f||_{C^{k,\al}_{\delta}} = ||f||_{C^{k,\al}_{\omega}(M\setminus B_1)} + \sup_{\vep\leq r\leq 1/2}r^{-\delta}||f||_{C^{k,\al}_{r^{-2}\omega_{\vep}}(B_{2r}\setminus B_r)} + \vep^{-\delta}||f||_{C^{4,\al}_{\eta}(B_\vep)},$$ and let $C^{k,\al}_\delta$ denote the corresponding Banach space. As before we denote the $T$-invariant part by $(C^{k,\al}_{\delta})^T$. Note that as sets,  $C^{k,\al}_\delta$ and $C^{k,\al}$ are identical. As pointed out in \cite{Sz12, Sz15}, a more useful way of thinking of these norms is that $||f||_{C^{k,\al}_\delta} < C$ if 
\begin{align*}
|\nabla^jf| &< C,~r>1\\
|\nabla^j f| &< Cr^{\delta-j},~ \vep\leq r\leq1\\
|\nabla^j f|&<C\vep^{\delta-j},~ r\leq\vep.
\end{align*}

In our estimates we will use the following two important facts. These are not difficult to verify. 
\begin{align*}
||f||_{C^{k,\al}_{\delta'}} &\leq \begin{cases} ||f||_{C^{k,\al}_\delta},~ \delta'\leq \delta\\ 
\vep^{\delta-\delta'}||f||_{C^{k,\al}_\delta},~\delta<\delta'. \end{cases}\\
||fg||_{C^{k,\al}_{\delta + \delta'}} &\leq ||f||_{C^{k,\al}_\delta}||g||_{C^{k,\al}_{\delta'}}.
\end{align*}
The second inequality is especially useful to control nonlinear terms. We have the following estimate on the norm of the lifting operator $\bl$.
\begin{lem}\label{estimate on lift}For any $f\in \overline{\mathfrak{h}}$, $$||\bl(f)||_{C^{1,\al}_0}\leq C|f|,$$ where $|f|$ is any norm on the finite dimensional vector space $\overline{\mathfrak{h}}$.
\end{lem}

We now state some fundamental estimates which will allow us to control various linear and nonlinear terms in weighted Holder spaces. The proof of the following estimate can be found in \cite{Sz12}
\begin{lem}\label{linear terms estimate}
There exist constants $c_0$ and $C$ such that $$||\varphi||_{C^{4,\al}_2} < c_0 \implies ||\sL_{\omega_\varphi}(f) - \sL_{\oep}(f)||_{C^{0,\al}_{\delta-4}} \leq C||\varphi||_{C^{4,\al}_2}||f||_{C^{4,\al}_{\delta}}.$$
\end{lem}

Next, we need the following generalization of Lemma 21 in \cite{Sz12}.

\begin{lem}\label{nonlinear terms estimate}
There exist constants $c_0$, $C$ and $C'$ such that $$||Q_{\oep}^{(k)}(\varphi) - Q_{\oep}^{(k)}(\psi)||_{C^{0,\al}_{\delta-4}} \leq Ck^3(||\varphi||_{C^{4,\al}_2} + ||\psi||_{C^{4,\al}_2})^{k-1}||\varphi-\psi||_{C^{4,\al}_{\delta}},$$ and if $||\varphi||_{C^{4,\al}_2}, ||\psi||_{C^{4,\al}_2}<c_0$, then $$||Q_{\oep}^{(\geq k)}(\varphi) - Q_{\oep}^{(\geq k)}(\psi)||_{C^{0,\al}_{\delta-4}} \leq C'(||\varphi||_{C^{4,\al}_2} + ||\psi||_{C^{4,\al}_2})^{k-1}||\varphi-\psi||_{C^{4,\al}_{\delta}}.$$ 
\end{lem}

\begin{proof}
First note that if $\vep$ is sufficiently small, if we write $g_\vep$ for the metric corresponding to $\oep$, then $$||g_\vep||_{C^{2,\al}_0},||g_\vep^{-1}||_{C^{2,\al}_0}<C,$$ for constant $C$ independent of $\vep$. Moving on to a proof of the Lemma, first consider the case when $k=3$. A general term in $Q^{(3)}(\varphi)$ has the form $$\nabla^a\varphi\nabla^b\varphi\nabla^c\varphi,$$ where $a+b+c = 8$, and each index is smaller than or equal to $4$. Subtracting the analogous term for $\psi$, we have 
\begin{align*}
||\nabla^a\varphi\nabla^b\varphi\nabla^c\varphi-\nabla^a\psi\nabla^b\psi\nabla^c\psi||_{C^{0,\al}_{\delta-4}} \\
= ||\nabla^a\varphi\nabla^b\varphi\nabla^c\varphi -  \nabla^a\varphi\nabla^b\varphi\nabla^c\psi||_{C^{0,\al}_{\delta-4}} &+  ||\nabla^a\varphi\nabla^b\varphi\nabla^c\psi - \nabla^{a}\varphi\nabla^b\psi\nabla^c\psi||_{C^{0,\al}_{\delta-4}}\\ 
&+||\nabla^a\varphi\nabla^b\psi\nabla^c\psi - \nabla^{a}\psi\nabla^b\psi\nabla^c\psi||_{C^{0,\al}_{\delta-4}}.\\ 
\end{align*}
We estimate the first term. 
\begin{align*}
 ||\nabla^a\varphi\nabla^b\varphi\nabla^c\varphi -  \nabla^a\varphi\nabla^b\varphi\nabla^c\psi||_{C^{0,\al}_{\delta-4}} &\leq C||\nabla^{a}\varphi\nabla^b\varphi||_{C^{0,\al}_{ -4 + c}}||\nabla^c(\varphi-\psi)||_{C^{4,\al}_{\delta -c}}\\
 &\leq C||\nabla^{a}\varphi\nabla^b\varphi||_{C^{0,\al}_{4 -a-b}}||\varphi-\psi||_{C^{4,\al}_{\delta}}\\
 &\leq C||\nabla^a\varphi||_{C^{0,\al}_{2-a}}|\nabla^b\varphi||_{C^{0,\al}_{2-b}}||\varphi-\psi||_{C^{4,\al}_{\delta}}\\
 &\leq C||\varphi||^2_{C^{4,\al}_2}||\varphi-\psi||_{C^{4,\al}_{\delta}}.
\end{align*}
The other two terms can be estimated in a similar way. For a general $k$, we would have terms of the sort $$\nabla^{a_1}\varphi\cdots\nabla^{a_k}\varphi,$$ where $a_1+\cdots a_k = 2(k+1)$. Just as above, we can now break up the difference of two such terms into $k$ differences, and obtain an estimate of the form $$||\nabla^{a_1}\varphi\cdots\nabla^{a_k}\varphi - \nabla^{a_1}\psi\cdots\nabla^{a_k}\psi||_{C^{0,\al}_{\delta-4}} <Ck(||\varphi||_{C^{4,\al}_2} + ||\psi||_{C^{4,\al}_2})^{k-1}||\varphi-\psi||_{C^{4,\al}_{\delta}}.$$ The number of solutions to $a_1+\cdots a_k = 2(k+1)$ is bounded by $k^2$, and so adding up all the contributions $$||Q_{\oep}^{(k)}(\varphi) - Q_{\oep}^{(k)}(\psi)||_{C^{0,\al}_{\delta-4}} \leq Ck^3(||\varphi||_{C^{4,\al}_2} + ||\psi||_{C^{4,\al}_2})^{k-1}||\varphi-\psi||_{C^{4,\al}_{\delta}}$$
If $c_0$ is small enough, we can obtain the second estimate from the first by simply summing up a convergent series.
\end{proof}
We also need a version of Lemma 19 in \cite{Sz15}.
\begin{lem}\label{nonlinear estimate 2}
Let $\omega = \oep + \ddbar \varphi$. There exists a constant $c_0$ such that if $$||\varphi||_{C^{4,\al}_2},||\psi||_{C^{4,\al}_2}<c_0,$$ then $$||Q^{(\geq3)}_{\omega}(\psi) - Q^{(\geq3)}_{\oep}(\psi)||_{C^{0,\al}_{\delta-4}} \leq C||\varphi||_{C^{4,\al}_2}||\psi||^2_{C^{4,\al}_2}||\psi||_{C^{4,\al}_\delta}.$$
\end{lem}
\begin{proof}
Schematically, if $g$ is a metric and $h$ a small perturbation, $$Q_g^{(\geq 3)}(h) = \sum_{j=0}^{2}g^{-1}[\partial^j(g^{-1}h)^3]F_j(g^{-1}h),$$ where $F_j$ are formal power series. Note that the exponent three on $g^{-1}h$ in the sum above corresponds to the fact that we are only considering nonlinear terms of order higher than three. The proof is then identical to the one in \cite{Sz15}. The exponent of three results in the estimate depending on $||\psi||_{C^{4,\al}_2}^2$ instead of on $||\psi||_{C^{4,\al}_2}$ as in \cite{Sz15}. 
\end{proof}

\begin{remark}
Even though the results are stated for the metric $\oep$, they also hold on domains such as $M_p\setminus B_{r_\vep}$ with the metric $\omega$, and on domains of the type $B_{R_\vep}$ in $\bC^2$ with the metric $\eta$. The results also hold for small perturbations of $\oep$. 
\end{remark}

\subsection{Inverting the linearized operator} The following is Proposition 22 in \cite{Sz12}.
\begin{prop}\label{lin1}
Let $\delta \in (-1,0)$ such that $|\delta|<<1$. Then the operator 
\begin{align*}
\sF_1:C^{4,\al}_{-\delta}(\bM)^T\times \overline{\mathfrak{h}}&\rightarrow C^{0,\al}_{-\delta-4}(\bM)^T\\
(u,f)&\mapsto -\sD_{\oep}^*\sD_{\oep} u - \bl(f)
\end{align*}
has a right inverse $P_1$ with the bound $$||P_1|| \leq C\vep^{-\delta}.$$
\end{prop}
Note that although the definition of lifts is slightly different in \cite{Sz12}, the same proof goes through. Alternatively, one could use a standard blow-up argument similar to the ones in \cite[Theorem 8.14]{Szbook} and \cite[pg.~ 11]{SySz16}. Notice that the bound for the inverse blows up (albeit mildly) as $\vep\rightarrow 0$. It is possible to obtain a uniform bound if we restrict the range to functions with vanishing integrals. The next proposition is stated for dimensions greater than two in \cite[Proposition 21]{Sz15}. But it is clear from the proof that the proposition also holds in dimension two. 

\begin{prop}\label{lin2}Let $\overline{\mathfrak{h}}_0$ denote the subspace of $\overline{\mathfrak{h}}$ whose lifts have zero means with respect to $\oep$, and let $C^{0,\al}_{-\delta-4}(\bM)^T_0$ denote the space of functions with zero mean. Then if $\delta\in (-1,0)$, the operator 
\begin{align*}
\sF_2:C^{4,\al}_{-\delta}(\bM)^T\times \overline{\mathfrak{h}}_0&\rightarrow C^{0,\al}_{-\delta-4}(\bM)^T_0\\
(u,f)&\mapsto -\sD_{\oep}^*\sD_{\oep} u - \bl(f)
\end{align*}
has a uniformly bounded right inverse.
\end{prop}

As noted in \cite[Remark 22]{Sz15}, both these results continue to hold for $\Omega = \oep + \ddbar u$, so long as $||u||_{C^{4,\al}_2}$ is sufficiently small.

\section{Approximate solutions}

\subsection{Modifying the metric on $M_p$.} The strategy is similar to \cite{Sz12,Sz15} with one notable complication. Cutting off the log-term in the definition of $\oep$ introduces a larger error than what is needed. Hence one needs to introduce a log term in $\omega$. For this, let $G$ be the solution to the equation 
\begin{equation}\label{eqn:G}
-\sD_\omega^*\sD_\omega G - h_1 = 2\pi^2\delta_p,
\end{equation} for some $h_1\in \overline{\mathfrak{h}}$. This is of course a Green's function for the operator $-\sD_\omega^*\sD_\omega$, with singularity at $p$. It then follows from general elliptic theory that $$G(z) =\log|z| + l(z) + \tilde G(z),$$ for some linear function $l(z)$ and some $\tilde G \in C^{4,\al}_{2-\tau}(M_p)^T$ for all $\tau>0$. We could normalize so that $l(p) = 0$, and write $$l(z) = b_\al z^\al +\overline{b_\al z^\al}.$$ As in \cite{Sz12}, by integrating, we see that $$h_1 = -2\pi^2(V^{-1} + \mu(p)),$$ where $V = Vol(M,\omega)$. To construct a better approximate solution we add $\ddbar\vep^2\ga_1G$ to $\oep$. The highest order error on $M\setminus B_1$ then comes from $Q_\omega(\vep^2G)$ which is of order $O(\vep^4)$. Since we need a slightly smaller error, we need to add an additional correction term. The idea is to use Lemma  \ref{lin M_p}.

\begin{lem}\label{order} With $G$ as above, $Q^{(2)}_\omega(G)\in C^{0,\al}_{-4-\tau}(M_p)$, for all $\tau>0$.
\end{lem}
\begin{proof}
It is enough to show that the $C^{4,\al}_{-4-\tau}$ norm of  $Q^{(2)}_\omega(G)$ is bounded on $B_1\setminus\{p\}$ for all $\tau>0$. On this region, $$\vep^2\eta = \omega_{Euc} + \vep^2\ddbar(G - \tilde G).$$ Since $\vep^2\eta$ is scalar flat, it follows that $Q^{(2)}_0(G-\tilde G) = 0$. Then by Lemmas \ref{nonlinear terms estimate} and \ref{nonlinear estimate 2}, for any $\tau>0$, 
\begin{align*}
||Q^{(2)}_\omega(G)||_{C^{0,\al}_{-4-\tau}} &\leq ||Q^{(2)}_\omega(G) - Q_\omega^{(2)}(G-\tilde G)||_{C^{0,\al}_{-4-\tau}} + ||Q^{(2)}_\omega(G-\tilde G) - Q^{(2)}_0(G-\tilde G)||_{C^{0,\al}_{-4-\tau}}\\
&\leq C||G||_{C^{4,\al}_{2}}||\tilde G||_{C^{4,\al}_{-\tau}} + C||\varphi_4||_{C^{4,\al}_2}||G||_{C^{4,\al}_{2}}||G||_{C^{4,\al}_{-\tau}}\\
&\leq Cr^{\tau}(\log{r})^2\leq C.
\end{align*}

\end{proof}

\begin{prop}\label{residue} 
For every $f\in \overline{\mathfrak{h}}$ with $f(p) = 0$, $$\int_{M_p}[Q_\omega^{(2)}(G) - \mathrm{Re}(\langle \nabla h_1,\nabla G)]f~\frac{\omega^2}{2} = 4\pi ^2\mathrm{Re}(b_{\bal}f_\al(p)).$$
\end{prop}
In the proof of the above proposition we will repeatedly use the following standard integration by parts formula.  
\begin{lem}
For a $(0,1)$ form $\al$, and a domain $U\subset M$, $$\int_U \mathrm{div}(\al)\,d\mu_\omega = \int_{\partial U}\al_{\bar j}\nu^{\bar j}\,d\sigma_\omega,  $$ where $\mathrm{div}(\al) :=(\al_{\bar j})_{;j}$, $\nu$ is the internal normal vector, and $d\sigma_\omega = i_\nu d\mu_\omega$ is the corresponding surface measure.
\end{lem}

It is not difficult to see that the residue should be a combination of a term involving $\Delta f(p)$ and $b_{\bal}\nabla^{\bal}f(p)$. The crucial point is that the $\Delta f(p)$ term cancels out.

\begin{proof}[Proof of Proposition]
For $\rho \in (0,1)$, let $B_\rho =\{|z|<\rho\}$, and $M_\rho = M\setminus\overline{B_\rho}$. We denote any term that approaches zero as $\rho\rightarrow 0$ by $\vep(\rho)$. We also denote the unit outward normal to $B_\rho$ by $\nu$, and the induced measure on $\partial B_\rho$ by $d\sigma_\rho$, so that $$d\sigma_\rho = i_\nu\frac{\omega^2}{2}\Big|_{\partial B_\rho}.$$ Note that in normal coordinates 
\begin{align*}
g_{\al\bbe} &= \frac{\delta_{\al\bbe}}{2} + O(|z|^2)\\
\nu &= \frac{z^\al}{|z|}\frac{\partial}{\partial z^\al}+\frac{z^{\bbe}}{|z|}\frac{\partial}{\partial z^{\bbe}},\\
d\sigma_\rho &= \rho^3(1+O(\rho^2))d\theta,
\end{align*} where $d\theta$ is the standard Lebesgue measure on $\mathbb{S}^3$.
Since $Q_\omega^{(2)}(G) \in C^{0,\alpha}_{-4}(M_p)$, it is enough to compute $$\lim_{\rho\rightarrow 0^+}\int_{M_\rho}[Q_\omega^{(2)}(G) - \mathrm{Re}(\langle\nabla h,\nabla G\rangle)]f\frac{\omega^2}{2!}.$$ From equation \eqref{Q formula}, we can write $$2I(\rho) := \int_{M_\rho}2Q_\omega^{(2)}(G)f\frac{\omega^2}{2} = 2I_1(\rho) + I_2(\rho) + 2I_3(\rho),$$ where 
\begin{align*}
I_1(\rho) &= \int_{M_\rho} \nabla^\al\nabla^{\bbe}G\nabla_\al\nabla_{\bbe}\Delta G \cdot f,\\
I_2(\rho) &=\int_{M_\rho}\Delta|\nabla\bar\nabla G|^2\cdot f,\\
I_3(\rho) &= \int_{M_\rho}R_{\al\bbe}\nabla^\al\nabla^\lambda G\nabla_\la\nabla^{\bbe} G\cdot f.
\end{align*}
Note that all the three integrals above  are real valued. For the first integral, integrating by parts $$I_1(\rho) = -\int_{M_\rho}\nabla^{\bbe}\Delta G \nabla_{\bbe}\Delta G \cdot f - \int_{M_\rho}\nabla^{\al}\nabla^{\bbe}G\nabla_{\bbe}\Delta G\nabla_\al f + \int_{\partial B_\rho}\nabla_{\bar\mu}\nabla^{\bbe}G\nabla_{\bbe}\Delta G\nu^{\bar\mu}f.$$ The Hessian of $G$ has the expansion $$\nabla_{\al}\nabla_{\bbe} G = \frac{1}{2}\Big(\frac{\delta_{\al\bbe}}{|z|^2} - \frac{\overline{z}^{\al}z^\beta}{|z|^4}\Big) + O(|z|^{-\tau}).$$
From this it can be seen easily that $$\nabla_{\bar\mu}\nabla_{\al}G\nabla^{\al}\Delta G\nu^{\bar\mu} = O(|z|^{-3}).$$ That is, the $O(|z|^{-5})$ term cancels out, and since $f = O(|z|)$, the boundary has only an $\vep(\rho)$ contribution, and $$I_1(\rho) = -\int_{M_\rho}\nabla^{\bbe}\Delta G \nabla_{\bbe}\Delta G \cdot f - \int_{M_\rho}\nabla^{\al}\nabla^{\bbe}G\nabla_{\bbe}\Delta G\nabla_\al f + \vep(\rho).$$
Integrating the first term by parts, $$-\int_{M_\rho}\nabla^{\bbe}\Delta G \nabla_{\bbe}\Delta G \cdot f = \int_{M_\rho}\Delta G \Delta^2 G \cdot f  + \int_{M_\rho}\Delta G \nabla_{\bbe}\Delta G \nabla^{\bbe}f - \int_{\partial B_\rho}\Delta G\nabla_{\bbe}\Delta G \nu^{\bbe}f.$$ For the boundary term, we notice that 
\begin{align*}
\Delta G &= |z|^{-2} + O(|z|^{-\tau})\\
\nabla_{\bbe}\Delta G &= -\frac{z^{\be}}{|z|^4} + O(|z|^{-2-\tau})\\
\nu^{\bbe} &= \frac{\bar z^\beta}{|z|},
\end{align*}
and so $$\Delta G\nabla_{\bbe}\Delta G \nu^{\bbe} = -|z|^{-5} + O(|z|^{-3}).$$ Now $f$ has the expansion $$f = f_{j}(p)z^j + f_{\bar j}(p)\bar z^j + \mathrm{Re}(f_{ij}(p)z^iz^j) + f_{i\bar j}(p)z^i\bar z^j + O(|z|^3).$$In fact there will be no $f_{ij}$ terms since $f$ is the potential of a holomorphic vector field. Here, and in all subsequent calculations, we will use the Einstein summation convention. Using the fact that linear functions and quadratic functions of the form $z^iz^j$ or $\bar z^i\bar z^j$ integrate out to zero on $\mathbb{S}^3$ and the following formula $$\int_{\mathbb{S}^3}z^{\al}\bar z^\be = \begin{cases} \pi^2,~ \al = \be \\ 0,~ \text{otherwise} \end{cases},$$ it is easy to see that $$\int_{\partial B_\rho}\Delta G\nabla_{\bbe}\Delta G \nu^{\bbe}f = -\frac{\pi^2}{2}\Delta f(p) + \vep(\rho),$$ and so $$I_1 = \int_{M_\rho}\Delta G \Delta^2 G \cdot f  + \int_{M_\rho}\Delta G \nabla_{\bbe}\Delta G \nabla^{\bbe}f - \int_{M_\rho}\nabla^{\al}\nabla^{\bbe}G\nabla_{\bbe}\Delta G\nabla_\al f  +\frac{\pi^2}{2}\Delta f(p) + \vep(\rho). $$ Next, we focus on the second term above. Again integrating by parts, and computing the boundary term as above,
\begin{align*}
\int_{M_\rho} \Delta G \nabla_{\bbe}\Delta G \nabla^{\bbe}f &= -\int _{M_\rho}\Delta G \nabla_{\bbe}\Delta G \nabla^{\bbe}f - \int_{M_\rho}(\Delta G)^2\Delta f + \int_{\partial B_\rho} (\Delta G)^2 \nabla_{\al}f \nu^{\al}\\
&=-\int _{M_\rho}\Delta G \nabla_{\bbe}\Delta G \nabla^{\bbe}f - \int_{M_\rho}(\Delta G)^2\Delta f + \frac{\pi^2}{2}\Delta f(p) + \vep(\rho),
\end{align*}
 and so $$\int_{M_\rho} \Delta G \nabla_{\bbe}\Delta G \nabla^{\bbe}f = -\frac{1}{2} \int_{M_\rho}(\Delta G)^2\Delta f + \frac{\pi^2}{4}\Delta f(p) + \vep(\rho).$$

Combining with the above formulae, we obtain that  $$I_1 =  \int_{M_\rho}\Delta G \Delta^2 G \cdot f-\frac{1}{2} \int_{M_\rho}(\Delta G)^2\Delta f - \int_{M_\rho}\nabla^{\al}\nabla^{\bbe}G\nabla_{\bbe}\Delta G\nabla_\al f  +\frac{3\pi^2}{4}\Delta f(p) + \vep(\rho).$$ For the third integral, $$ - \int_{M_\rho}\nabla^{\al}\nabla^{\bbe}G\nabla_{\bbe}\Delta G\nabla_\al f = \int_{M_\rho}\nabla^\al G\Delta^2 G \nabla_\al f - \int_{\partial B_\rho}\nabla^\al G\nabla_{\bbe}\Delta G\nu^{\bbe}\nabla_{\al}f.$$ To compute the boundary term, we note the following expansions
\begin{align*}
\nabla^\al G &= \frac{z^\al}{|z|^2} + 2b_{\bal} + O(|z|)\\
\nabla_{\bbe}\Delta G &= -\frac{z^{\be}}{|z|^4} + O(|z|^{-1-\tau})\\
\nu^{\bbe} &= \frac{\bar z^{\be}}{|z|}\\
\nabla_\al f &= f_\al(p) + f_{\al\la}(p)z^\la + f_{\al\bar\nu}\bar z^\nu.
\end{align*} 
The only terms that will contribute a non-zero boundary integrals are $$-\frac{f_{\al\bar\nu}(p)z^\al\bar z^\nu}{|z|^5} - \frac{2b_\al f_{\bal}(p)}{|z|^3}.$$ Integrating, and using the fact that $|\mathbb{S}| = 2\pi^2$, we obtain that $$- \int_{\partial B_\rho}\nabla^\al G\nabla_{\bbe}\Delta G\nu^{\bbe}\nabla_{\al}f = \frac{\pi^2}{2}\Delta f(p) + 4\pi^2b_\al f_{\bal}(p) +\vep(\rho),$$ and so 
\begin{align*}
I_1(\rho) =  \int_{M_\rho}\Delta G \Delta^2 G \cdot f &+ \int_{M_\rho}\nabla^\al G\Delta^2 G \nabla_\al f  -\frac{1}{2} \int_{M_\rho}(\Delta G)^2\Delta f  \\
& +\frac{5\pi^2}{4}\Delta f(p) +4\pi^2b_\al f_{\bal}(p)+ \vep(\rho),
\end{align*}
For the first two terms since $\Delta^2G$ is only of the order $O(|z|^{-2})$, we can freely integrate by parts, and we finally obtain (after taking the real part)
\begin{align}\label{I_1}
I_1(\rho) &= - \mathrm{Re}\Big(\int_{M_\rho}\nabla^\al G \nabla_\al\Delta^2 G \cdot f \Big)  -\frac{1}{2} \int_{M_\rho}(\Delta G)^2\Delta f   \\\nonumber
&+\frac{5\pi^2}{4}\Delta f(p) +4\pi^2\mathrm{Re}(b_\al f_{\bal}(p))+ \vep(\rho).
\end{align}

Next we move on to the second integral. $$I_2(\rho) = \int_{M_\rho}|\nabla\overline\nabla G|^2 \Delta f + \frac{1}{2}\int_{\partial B_\rho}\partial_\nu|\nabla\overline\nabla G|^2 \cdot f - \frac{1}{2}\int_{\partial B_\rho}|\nabla\overline\nabla G|^2 \partial_\nu f.$$
Note that the half in the boundary terms is due to the fact that the complex Laplacian is half that of the real Laplacian. We then compute the boundary terms. Note that $$|\nabla\overline\nabla G|^2 = |z|^{-4} + O(|z|^{-2}),$$ and so $$\frac{1}{2}\int_{\partial B_\rho}\partial_\nu|\nabla\overline\nabla G|^2 \cdot f  = -\pi^2\Delta f(p) + \vep(\rho),$$ and  
\begin{align*}
\frac{1}{2}\int_{\partial B_\rho}|\nabla\overline\nabla G|^2 \partial_\nu f &= \frac{\rho^{-4}}{2}\int_{\partial B_\rho}\partial_\nu f +\vep(\rho)\\
&=\rho^{-4}\int_{B_\rho}\Delta f +\vep(\rho)\\
&= \omega_4\Delta f(\rho) =\frac{\pi^2}{2}\Delta f(p).
\end{align*}
Here $\omega_4$ denotes the volume of the unit ball in $\CC^2$. Combining, we get that 
\begin{equation}\label{eqn:I2-1}
I_2(\rho) = \int_{M_\rho}|\nabla\overline\nabla G|^2 \Delta f -\frac{3\pi^2}{2}\Delta f(p) + \vep(\rho).
\end{equation}  For the first integral, $$\int_{M_\rho}|\nabla\overline\nabla G|^2 \Delta f = -\int_{M_\rho}\nabla_\al\Delta G \nabla^\al G\Delta f - \int_{M_\rho}\nabla_\al\nabla_{\bbe}G \nabla^\al G \nabla^{\bbe}\Delta f + \int_{\partial B_\rho}\nabla_\al\nabla_{\bbe}G\nabla^\al G\nu^{\bbe}\Delta f.$$ It is not difficult to see that $$\nabla_\al\nabla_{\bbe}G\nabla^\al G\nu^{\bbe} = O(|z|^{-2}),$$ since $O(|z|^{-3})$ contribution cancels out. So the boundary term will not contribute any residue and we have that $$\int_{M_\rho}|\nabla\overline\nabla G|^2 \Delta f = -\int_{M_\rho}\nabla_\al\Delta G \nabla^\al G\Delta f - \int_{M_\rho}\nabla_\al\nabla_{\bbe}G \nabla^\al G \nabla^{\bbe}\Delta f  + \vep(\rho).$$ For the first integral above, 
\begin{align*}
\int_{M_\rho}\nabla_\al\Delta G \nabla^\al G\Delta f &=- \int_{M_\rho}(\Delta G)^2\Delta f - \int_{M_\rho}\Delta G \nabla^\al G \nabla_\al\Delta f + \int_{\partial B_\rho}\Delta G \nabla_{\bbe}G\nu^{\bbe}\Delta f \\
&=- \int_{M_\rho}(\Delta G)^2\Delta f - \int_{M_\rho}\Delta G \nabla^\al G \nabla_\al\Delta f  + \pi^2\Delta f(p) + \vep(\rho), 
\end{align*} 
and so 
\begin{align*}
\int_{M_\rho}|\nabla\overline\nabla G|^2 \Delta f = \int_{M_\rho}(\Delta G)^2\Delta f + \int_{M_\rho}\Delta G \nabla^\al G \nabla_\al\Delta f &- \int_{M_\rho}\nabla^\al\nabla^{\bbe}G \nabla_{\bbe} G \nabla_{\al}\Delta f  \\
&- \pi^2\Delta f(p)+ \vep(\rho).
\end{align*}
Commuting the derivatives, and using the fact that $f$ is the Hamiltonian of a holomorphic vector field we see that $$\nabla_\al\Delta f = -R_{\al\bar\mu}\nabla^{\bar \mu}f, $$ and so we can write 
\begin{equation}\label{eqn:I2-2}
\int_{M_\rho}|\nabla\overline\nabla G|^2 \Delta f  = \int_{M_\rho}(\Delta G)^2\Delta f + J_1+J_2 -\pi^2\Delta f(p) + \vep(\rho),
\end{equation} where 
$$J_1 = -\int_{M_\rho}\Delta G \nabla^\al G~R_{\al\bar\mu}\nabla^{\bar \mu}f,~J_2 = \int_{M_\rho}\nabla^\al\nabla^{\bbe}G \nabla_{\bbe} GR_{\al\bar\mu}\nabla^{\bar\mu}f.$$ 
Integrating by parts, using the Bianchi identity and the fact that $\omega$ has constant scalar curvature, we obtain that $$J_2 = -I_3(\rho) - \int_{M_\rho}\nabla_{\bbe} G\nabla^{\bar\mu}\nabla^\al \nabla^{\bbe} G R_{\al\bar\mu}f + \int_{\partial B_\rho}\nabla^\al G\nabla_\al\nabla^\tau GR_{\tau\bar\mu}\nu^{\bar\mu}f.$$  It is easy to see that the boundary term is $\vep(\rho)$ since the integrand if $O(|z|^{-2})$, and so $$J_2 = -I_3(\rho) -\int_{M_\rho}\nabla_{\bbe} G\nabla^{\bar\mu}\nabla^\al \nabla^{\bbe} G R_{\al\bar\mu}f +\vep(\rho).$$ 
We now turn our attention to $J_1$. Integrating by parts we obtain
$$J_1 = \int_{M_\rho}\nabla^{\bar\mu}\Delta G \nabla^\al G R_{\al\bar\mu}f + \int_{M_\rho}\Delta G(R_{\al\bar\mu}\nabla^\al\nabla^{\bar \mu}G)f -\int_{\partial B_\rho}\Delta G \nabla^{\al}G R_{\al\bar\mu}\nu^{\bar \mu}f.$$ Once again it is easy to see that the boundary term is $\vep(\rho)$. In fact, since all the subsequent boundary terms will be of a similar order to the above term, we will not bother writing them out and will simply label them as $\vep(\rho)$. Writing $\Delta G = \nabla_\tau\nabla^\tau G$, and integrating the first term by parts we obtain, 
\begin{align*}
J_1 &= -I_3(\rho) - \int_{M_\rho}\nabla^{\bar\mu}\nabla^\tau G \nabla^\al G\nabla_{\tau}R_{\al\bar\mu}f \\
&- \int_{M_\rho}\nabla^{\bar\mu}\nabla^\tau G \nabla^\al GR_{\al\bar\mu}\nabla_\tau f + \int_{M_\rho}\Delta G(R_{\al\bar\mu}\nabla^\al\nabla^{\bar \mu}G)f + \vep(\rho)\\
&= -I_3(\rho) - \int_{M_\rho}\nabla^{\bar\mu}\nabla^\tau G \nabla^\al G\nabla_{\tau}R_{\al\bar\mu}f + \int_{M_\rho}(R_{\al\bar\mu}\nabla^\al\nabla^{\bar \mu}G)(f\Delta G + \nabla^\tau G\nabla_\tau f) + \vep(\rho) \\
&=-I_3(\rho) - \int_{M_\rho}\nabla^{\bar\mu}\nabla^\tau G \nabla^\al G\nabla_{\tau}R_{\al\bar\mu}f -  \int_{M_\rho}\nabla^\tau G\nabla_\tau(R_{\al\bar\mu}\nabla^\al\nabla^{\bar \mu}G) + \vep(\rho).
\end{align*}
For the second term above, using the fact that $\nabla_\tau R_{\al\bar\mu} = \nabla_\al R_{\tau\bar\mu}$,
\begin{align*}
- \int_{M_\rho}\nabla^{\bar\mu}\nabla^\tau G \nabla^\al G\nabla_{\tau}R_{\al\bar\mu}f &= -I_3(\rho) - J_2 + \int_{M_\rho}(R_{\tau\bar\mu}\nabla^\tau\nabla^{\bar\mu}G)(f\Delta G + \nabla^\al G\nabla_\al f) + \vep(\rho)\\
&= -I_3(\rho) - J_2 - \int_{M_\rho}\nabla^\al G\nabla_\al(R_{\tau\bar\mu}\nabla^\tau\nabla^{\bar\mu}G)f +\vep(\rho), 
\end{align*}
and so $$J_1+J_2 = -2I_3(\rho)-2 \int_{M_\rho}\nabla^\al G\nabla_\al(R_{\tau\bar\mu}\nabla^\tau\nabla^{\bar\mu}G)f + \vep(\rho).$$
Combining with \eqref{eqn:I2-1} and \eqref{eqn:I2-2}, we see that $$I_2(\rho) + 2I_3(\rho) =  \int_{M_\rho}(\Delta G)^2\Delta f -2 \int_{M_\rho}\nabla^\al G\nabla_\al(R_{\tau\bar\mu}\nabla^\tau\nabla^{\bar\mu}G)f  -\frac{5\pi^2}{2}\Delta f(p) + \vep(\rho).$$ Taking the real part, dividing by two and combining with \eqref{I_1} $$I_1(\rho)+\frac{1}{2}I_2(\rho) +I_3(\rho) = - \mathrm{Re}\Big(\int_{M_\rho} \nabla_\al(\sD_\omega^*\sD_\omega G)\nabla^\al G \cdot f \Big) + 4\pi^2\mathrm{Re}(b_\al f_{\bal}(p))+ \vep(\rho),$$ where we used the fact that $$\sD_\omega^*\sD_\omega G = \Delta^2 G + R_{\al\bar\mu}\nabla^\al\nabla^{\bar\mu}G.$$ Finally using equation \eqref{eqn:G} and letting $\rho\rightarrow 0$, we complete the proof of the Lemma.


\end{proof}

Note that if there was a $U(2)$-action fixing $p$, the Green's function would have no linear terms (as for instance is the case in \cite{GV16}), and the residue would trivially be zero. Since the $T$ in the main theorem could be $S^1$, we need to work a bit more.

\begin{lem}\label{modifying Green}
If there is a non-trivial $S^1$ action fixing $p$, possibly after modifying $G$ by a holomorphy potential,  the residue $$\int_{M_p}[Q_\omega^{(2)}(G) - \mathrm{Re}(\langle \nabla h_1,\nabla G)]f~\frac{\omega^2}{2} = 0 $$for every $f\in \overline{\mathfrak{h}}$ with $f(p) = 0$,
\end{lem}
\begin{proof}
Note that the $S^1$ action is linear and unitary in our coordinates, and hence can be diagonalized. If the $S^1$ acts on both factors $(z_1,z_2)$ non-trivially, then there would be no invariant linear functions, and the Green's function locally would have no linear term. The residue is then trivially zero. Hence we can assume that the $S^1$ action is locally given by $$e^{i\theta}\cdot(z_1,z_2) = (e^{i\lambda\theta}z_1,z_2).$$ So the generating vector field $\xi_0$ is given by $$\xi_0 = \lambda z_1\frac{\partial}{\partial z_1}.$$ Since $z_2$ is the only invariant linear function under this action, we have that $$G(z) = \log|z| + bz_2 +\overline{bz_2}+ O(|z|^{2-\tau}).$$There is at least one $ \xi \in \mathfrak{h}$ such that $\xi(p)\neq0$, or else the residue is trivially zero. Since $\mathrm{Im}(\xi)$ commutes with $\mathrm{Im}(\xi_0)$, an easy computation shows there is a $B\neq0$ such that, $$\xi(p) = B\frac{\partial}{\partial z_2},$$ that is, the $z^1$ component of $\xi(p)$ vanishes (only at $p$). If $h_\xi$ is the corresponding Hamiltonian, then locally at $p$, $$h_\xi(p) = Bz_2 + \overline{Bz_2} + O(|z|^2).$$  But then $$\tilde G = G - bB^{-1}h_\xi$$ is a $T$-invariant Green's function (since $h_\xi$ is killed by $\sD_\omega^*\sD_\omega$), and has no linear part locally. Using this Green's function, the integral then evaluates to zero.

\end{proof}

From now on we choose a Green's function for which the above integral vanishes. Now consider the equation 
\begin{equation}\label{eqn:correction}
\sD_\omega^*\sD_\omega \Phi = Q_\omega^{(2)}(G) -  \mathrm{Re}(\langle \nabla h_1,\nabla G). 
\end{equation} By Proposition \ref{residue} and Lemma \ref{modifying Green} the right hand side is orthogonal to all $f\in \overline{\mathfrak{h}}$ such that $f(p) = 0$. Moreover from Lemma \ref{order}, it follows that the right hand side is in $C^{0,\al}_{-4-\tau}$ for any $\tau>0$. Then it follows from Lemma \ref{lin M_p} that equation \eqref{eqn:correction} has a solution $\Phi \in C^{4,\al}_{-\tau}$. Note that the right hand side does not depend on $\vep$, and so the bound on $||\Phi||_{C^{4,\al}_{-\tau}(M_p)}$ is independent of $\vep$. It will of course depend on $\tau$, but we will choose a small enough $\tau$ later. Now let 
\begin{align}
\Ga &= \vep^2G - \vep^4\frac{\s_\omega}{4\pi}G + \vep^4\Phi \\\nonumber
f_1 &= \s_\omega - 2\pi^2\vep^2(V^{-1}+\mu(p)) + \frac{\pi\vep^4}{2}\s_\omega(V^{-1}+\mu(p)).
\end{align}  For small enough $\vep$, $\omega + \ddbar\Ga$ is a metric on $M\setminus B_{r_\vep}$. We need the extra term of order $O(\vep^4)$ involving $G$ and $\s_\omega$  to make a certain integral small. This would be useful to construct a more refined approximation in the next section.

\begin{lem}\label{s estimate region 1}
$$||\s(\omega + \ddbar\Ga) - f_1 - \mathrm{Re}\langle \nabla f_1,\nabla \Ga\rangle||_{C^{0,\al}_{-5}(M\setminus B_{2r_\vep})} = O(\vep^\kappa),$$ for some $\kappa>4$.
\end{lem}

\begin{proof}
Expanding and using the equations satisfied by $G$ and $\Phi$, we see that 
\begin{align*}
\s(\omega + \ddbar\Ga) - f_1 - \mathrm{Re}\langle \nabla f_1,\nabla \Ga\rangle &= Q_\omega(\Ga) - Q_\omega^{(2)}(\vep^2G)\\
& -\vep^2\mathrm{Re}\langle\nabla h_1,\nabla(\Ga - \vep^2G)\rangle + \vep^4\frac{\s_\omega}{4\pi}\mathrm{Re}\langle \nabla h_1,\nabla\Ga\rangle.
\end{align*}We now need to consider two regions.
\begin{itemize}
\item $|z|>1/2$. Here the weights are irrelevant. Since $\Ga-\vep^2G = O(\vep^4)$, and $\Ga = O(\vep^2)$ it is not difficult to see that the bottom two terms are of the order $O(\vep^6)$. So all we need to do is to estimate the non-linear terms. 
\begin{align*}
||Q_\omega(\Ga) - Q_\omega^{(2)}(\vep^2G)||_{C^{0,\al}} &\leq ||Q_\omega(\Ga) - Q_\omega(\vep^2G)||_{C^{0,\al}} + ||Q_\omega(\vep^2 G) - Q_\omega^{(2)}(\vep^2G)||_{C^{0,\al}}\\
&\leq C(||\Ga||_{C^{4,\al}} + ||\vep^2G_1||_{C^{4,\al}})||\Ga - \vep^2G_1||_{C^{4,\al}} + C||\vep^2G_1||_{C^{4,\al}}^3 \\
&= O(\vep^6).
\end{align*}
\item $2r_\vep<|z|<1/2.$ The estimate is similar to the one above, except we now have to worry about the weights. We estimate in annuli $$A_r = \{r<|z|<2r\},$$ where $r\in (r_\vep,1/4)$.
\begin{align*}
||Q_\omega(\Ga) - Q_\omega^{(2)}(\vep^2G)||_{C^{0,\al}_{-5}(A_r)} &\leq ||Q_\omega^{(\geq 3)}(\Ga)||_{C^{0,\al}_{-5}(A_r)} + ||Q^{(2)}_\omega(\Ga) - Q_\omega^{(2)}(\vep^2G)||_{C^{0,\al}_{-5}(A_r)}
\end{align*}
To estimate the second term, we use Lemma \ref{nonlinear terms estimate} with $k=2$. Since the leading terms in $\vep^2\Ga$ and $\vep2 G$ is $\vep^2\log|z|$, and the leading term in $\vep^2(\Ga-G)$ is $\vep^4\Phi$ with $||\Phi||_{C^{4,\al}_{-\tau}}<C$ for a small $\tau>0$, we have 
\begin{align*}
||Q^{(2)}_\omega(\Ga) - Q_\omega^{(2)}(\vep^2G)||_{C^{0,\al}_{-5}(A_r)} &\leq C||\vep^2\log|z|||_{C^{4,\al}_2(A_r)}||\vep^4\Phi||_{C^{4,\al}_{-1}(A_r)}\\
&\leq C\vep^6r^{-1-\tau} \\
&< C\vep^6r_\vep^{-1-\tau} = O(\vep^{\kappa})
\end{align*}
for some $\kappa>4$, since $\al<1$.

 To estimate the $Q_\omega^{(\geq 3)}(\Ga)$ term, we first  note that by Lemma \ref{nonlinear terms estimate} $$||Q_\omega^{(\geq 3)}(\Ga) - Q_\omega^{(\geq 3)}(\vep^2\log|z|+l(z))||_{C^{4,\al}_{-5}(A_r)} \leq C||\Ga||_{C^{4,\al}_2}^2||\vep^2\tilde G||_{C^{4,\al}_{-1}}<C\vep^6 r^{-1-\tau} = O(\vep^\kappa),$$ for some $\kappa>4$.
 
 So it is enough to estimate $||Q_\omega^{(\geq 3)}(\vep^2\log|z|+l(z))||_{C^{4,\al}_{-5}(A_r)}.$ Since $$\vep^2\eta = \omega_{Euc} + \vep^2\ddbar(\log|z| +l(z))$$ is the Burns-Simanca metric (note that adding the linear term $l(z)$ makes no difference since it is killed by $\ddbar$), and is scalar flat, it follows that $$Q_0^{(\geq 3)}(\vep^2\log|z|+l(z)) = 0.$$ Then by Lemma \ref{nonlinear estimate 2}, 
 \begin{align*}
 ||Q_\omega^{(\geq 3)}(\vep^2\log|z|+l(z))||_{C^{4,\al}_{-5}(A_r)} &= ||Q_\omega^{(\geq 3)}(\vep^2\log|z|+l(z))-Q_0^{(\geq 3)}(\vep^2\log|z|+l(z))||_{C^{4,\al}_{-5}(A_r)}\\
 &\leq ||\varphi_4||_{C^{4,\al}_2(A_r)}||\vep^2\log|z|||^2_{C^{4,\al}_{2}(A_r)}||\vep^2\log|z|||_{C^{4,\al}_{-1}(A_r)}\\
 &<C\vep^6r^{-1-\tau} = O(\vep^\kappa),
 \end{align*}
 for some $\kappa>4$. The terms involving $h_1$ are easier to estimate. For instance $$\vep^2||\langle\nabla h_1,\nabla(\Ga - \vep^2G_1)\rangle||_{C^{0,\al}_{-5}(A_r)} \leq \vep^2||\nabla h_1||_{C^{0,\al}_{-1}(A_r)}||\Ga- \vep^2G||_{C^{0,\al}_{-4}(A_r)}\leq C\vep^6|h_1|||\Phi||_{C^{4,\al}_{-\tau}} = O(\vep^6).$$
\end{itemize}
\end{proof}

\subsection{First approximatation} Consider the metric $$\Omega_1 = \begin{cases} \omega + \ddbar \Ga,~ |z|>2r_\vep\\ \omega+\vep^2\ddbar (\log|z| + l(z)) + \ddbar(\vep^2\ga_1\tilde G + \vep^4\ga_1(\Phi - \s_\omega G/4\pi)),~r_\vep<|z|<2r_\vep\\\oep,~|z|<r_\vep.\end{cases}$$

We write $\Omega_1 = \oep + \ddbar u_1$. We then have the following estimates. 

\begin{lem}\label{scal estimate damage}
There exists $\kappa>4$ such that $$||\s(\Omega_1) -\s_\omega||_{C^{0,\al}_{-5}(B_{2r_\vep})} = O(\vep^\kappa).$$
\end{lem}
\begin{proof}
There are three regions to consider. 
\begin{itemize}
\item $r_\vep<|z|<2r_\vep$. On this region, since $l(z)$ is killed by $\ddbar$, we can write $$\Omega_1 = \vep^2\eta + \ddbar(A_4+A_5+ G_1),$$ where $$G_1 = \vep^2\ga_1\tilde G_1  +\vep^4\ga_1(\Phi - \s_\omega G/4\pi)  +\varphi_6.$$ Using the fact that $-\Delta_0^2(A_4 + A_5) = \s_\omega$, we can write
\begin{align*}
\s(\Omega_1) -  \s_\omega = (L_{\vep^2\eta}+\Delta_0)(A_4 + A_5) + L_{\vep^2\eta}(\vep^2\ga_1 \tilde G_1 + \varphi_6) + Q_{\vep^2\eta}(\vep^2\ga_1\tilde G_1 + \varphi_4)
\end{align*}
By Lemma \ref{linear terms estimate}  we have the estimate $$||(L_{\vep^2\eta}+\Delta_0)(A_4 + A_5)||_{C^{0,\al}_{-5}}\leq C||\vep^2\log{|z|}||_{C^{4,\al}_2}||A_4||_{C^{4,\al}_{-1}} = O(\vep^2r_\vep^3\log\vep) =O(\vep^\kappa),$$ for some $\kappa>4$ since $\al>2/3$. In $G_1$ the highest order term is $\vep^2\ga_1\tilde G$, and so 
\begin{align*}
||L_{\vep^2\eta}(G_1)||_{C^{0,\al}_{-5}}&\leq C||\vep^2\ga_1 \tilde G||_{C^{4,\al}_{-1}} = O(\vep^2r_\vep^{3-\tau}) = O(\vep^\kappa)\\
||Q_{\vep^2\eta}(\varphi_4+\vep^2\ga_1\tilde G_1  )||_{C^{0,\al}_{-5}} &\leq C||A_4||_{C^{4,\al}_2}||A_4||_{C^{4,\al}_{-1}} = O(r_\vep^7) = O(\vep^\kappa),
\end{align*} for some $\kappa>4$.
Putting all the estimates together,  $$||\s(\Omega_1) - \s_\omega||_{C^{0,\al}_{-5}} = O(\vep^\kappa).$$ 

\item $|z|<r_\vep$. On this region $\Omega_1 = \oep$. It is more convenient to instead work with the metric $\vep^{-2}\oep$ on $\bC^2$. Like in the proof of Lemma \ref{zero approx}, we let $w = \vep^{-1}z$, so that $|z|<r_\vep$ can be thought of as the region $\tilde B_{R_\vep} = \{|w|<r_\vep\}\subset\bC^2$, where $R_\vep = \vep^{-1}r_\vep$. We first estimate on the region $1<|w|<R_\vep$. The metric looks like $$\vep^{-2}\oep = \eta + \ddbar \Psi, $$
where $$\Psi = \vep^2A_4(w) + \vep^3A_5(w) + \vep^{-2}\varphi_6(\vep w).$$ Once again using the fact that $-\Delta^2_0(A_4+A_5) = \s_\omega$, we have
\begin{align*}
||\s(\vep^{-2}\oep) - \vep^2\s_\omega||_{C^{0,\al}_{-5}(\eta)} \leq ||(\sL_\eta + \Delta_0^2)(\vep^2A_4 + \vep^3A_5)||_{C^{0,\al}_{-5}(\eta)} + ||Q_\eta(\Psi)||_{C^{0,\al}_{-5}(\eta)}.
\end{align*}
Since the leading term in $(\sL_\eta + \Delta_0^2)$ is $|w|^{-2}$, we can control the first term by $$||(\sL_\eta + \Delta_0^2)(\vep^2A_4 + \vep^3A_5)||_{C^{0,\al}_{-5}(\eta)} \leq \vep^2R_\vep^3 = O(\vep^{-1+3\al}) = O(\vep^{1+\tau}),$$ for some $\tau>0$, since $\al>2/3.$ Equivalently, we could have also used Lemma \ref{linear terms estimate}. For the nonlinear term, the principal contribution comes from $\vep^2A_4$. Now $$||Q_\eta(\vep^2A_4)||_{C^{0,\al}_2(\eta)} = \vep^2R_\vep^2 <<1,$$ and so  by  Lemma \ref{nonlinear terms estimate} $$||Q_\eta(\Psi)||_{C^{0,\al}_{-5}(\eta)} \leq C||\vep^2A_4||_{C^{4,\al}_2(\eta)}||\vep^2A_4||_{C^{4,\al}_{-1}(\eta)} = O(\vep^4R_\vep^7) = O(\vep^{1+\tau}),$$ for some $\tau>0$, since $\alpha>4/7$. Next, we estimate on the region $|w|<1$. Note that there are no weights in this region. From equation \eqref{estimate for psi} in the proof of Lemma \ref{zero approx}, we know that $$||\nabla^k\Psi|| = O(\vep^2),$$ and so $$||\s(\vep^{-2}\oep) -\vep^2 \s_\omega||_{C^{4,\al}} = O(\vep^2).$$ Combined with the above estimates we then have $$||\s(\vep^{-2}\oep) - \vep^2\s_\omega||_{C^{0,\al}_{-5}(B_{R_\vep},\eta)} = O(\vep^{1+\tau}),$$ for some $\tau>0$. But then transporting the estimates back to $\bM$, $$||\s(\oep) - \s_\omega||_{C^{0,\al}_{-5}(B_{r_\vep})} \leq \vep^5\vep^{-2}||\s(\vep^{-2}\oep) - \vep^2\s_\omega||_{C^{0,\al}_{-5}(B_{R_\vep},\eta)} = O(\vep^{4+\tau}),$$ for some $\tau>0$.
 \end{itemize}
\end{proof}

\begin{lem}\label{scal estimate first}
There exists $\kappa>4$ and a constant $C$ independent of $\vep$ such that $$||\s(\Omega_1) - \bl_{\Omega_1}(f_1)||_{C^{0,\al}_{-5}} \leq C\vep^{\kappa}. $$ 
\end{lem}

\begin{proof}
There are two regions to consider. 
\begin{itemize}
\item $|z|>2r_\vep$. The estimate on this region follows directly from Lemma \ref{s estimate region 1}, since on $|z|>2r_\vep$ $$\bl_{\Omega_1}(f_1) = f_1 + \mathrm{Re}\langle\nabla f_1,\nabla \Ga\rangle.$$

\item $|z|<2r_\vep$. By Lemma \ref{scal estimate damage}, 
\begin{align*}
||\s(\Omega) - \bl_{\Omega_1}(f_1)||_{C^{0,\al}_{-5}(B_{2r_\vep})} &\leq ||\s(\Omega_1) - \s_\omega||_{C^{0,\al}_{-5}(B_{2r_\vep})} + \vep^2 ||\bl_{\Omega_1}(h_1)||_{C^{0,\al}_{-5}(B_{2r_\vep})}\\
&\leq O(\vep^\kappa) +  \vep^2||\bl_{\Omega_1}(h_1)||_{C^{0,\al}_{-5}(B_{2r_\vep})} \\
&\leq  O(\vep^\kappa) + Cr_\vep^5\vep^2||\bl_{\Omega_1}(h_1)||_{C^{0,\al}_{0}} \leq O(\vep^\kappa) + Cr_\vep^5\vep^2 = O(\vep^\kappa),
\end{align*}
since $\al>2/5$.
\end{itemize}
\end{proof}

\subsection{Second approximation} As in \cite{Sz15}, we need to perturb this metric once more. 

\begin{lem}\label{integral estimate}There exists a $\kappa>4$ such that
$$\int_{\bM}[\s(\Omega_1) - \bl_{\Omega_1}(f_1)]\frac{\Omega^2_1}{2} = O(\vep^{\kappa}).$$
\end{lem}

\begin{proof}
It follows from algebro-geometric calculations (cf. \cite[Proposition 36]{Sz15})
\begin{align*}
\int_{\bM}\s(\Omega_1)\frac{\Omega_1^2}{2} &= \int_M\s_\omega\frac{\omega^2}{2} - 2\pi^2\vep^2\\\nonumber
\int_{\bM}\bl_{\Omega_1}(\s_\omega) \frac{\Omega_1^2}{2}&= \int_M\s_\omega\frac{\omega^2}{2} - \frac{\pi\vep^4\s_\omega}{2},
\end{align*}
and so
\begin{equation}\label{integral estimate for scal}
\int_{\bM}[\s(\Omega_1) - \bl_{\Omega_1}(\s_\omega)]\frac{\Omega^2_1}{2} = -2\pi^2\vep^2 +\frac{\pi\vep^4\s_\omega}{2}.
\end{equation}
The extra factors of $\pi$ in our case are due to the fact that the exceptional divisor has volume $2\pi$, as opposed to volume one in \cite{Sz15}. We now claim that for any $h\in \mathfrak{h}$, 
\begin{equation}\label{integral estimate for h}
\int_{\bM}\bl_{\Omega_1}(h)\frac{\Omega_1^2}{2} = \int_M h\frac{\omega^2}{2} + O(\vep^{\la})
\end{equation} for some $\la>2$ (in fact we can take $\la = 4$). If $h\in \mathfrak{t}$, then once again by Proposition 36 in \cite{Sz15} it follows that $$\int_{\bM}\bl_{\Omega_1}(h)\frac{\Omega_1^2}{2} = \int_M h\frac{\omega^2}{2} + O(\vep^4),$$ and so there is nothing to prove. So let us assume that $h\in \mathfrak{h}'.$ In particular $h(p) = 0$, $\bl(h) = h$ and $\bl_{\Omega_1}h = h + \mathrm{Re}\langle \nabla h,\nabla u_1\rangle$. Note that $$\frac{\Omega_1^2}{2} = \frac{\oep^2}{2} + \Delta u_1\frac{\oep^2}{2} + \frac{\ddbar u_1\wedge \ddbar u_1}{2}.$$ Using the fact that $$\int_{\bM}\mathrm{Re}\langle \nabla h,\nabla u_1\rangle \frac{\oep^2}{2}  + \int_{\bM} h\Delta u_1\frac{\oep^2}{2} =0,$$ and expanding, we have 
\begin{align*}
\int_{\bM}\bl_{\Omega_1}(h)\frac{\Omega_1^2}{2} &= \int_{\bM} (h + \mathrm{Re}\langle \nabla h,\nabla u_1\rangle)\frac{(\oep + \ddbar u_1)^2}{2}\\
&=\int_{\bM}h\frac{\oep^2}{2} + \int_{\bM}(h +  \mathrm{Re}\langle \nabla h,\nabla u_1\rangle)\frac{\ddbar u_1\wedge\ddbar u_1}{2}\\
&+\int_{\bM} \mathrm{Re}\langle \nabla h,\nabla  u_1\rangle\Delta u_1\frac{\oep^2}{2}
\end{align*}

We claim that the second and the third terms are both of order $O(\vep^{\la})$ for some $\la>2$. For instance to control $$\int_{\bM}h\frac{\ddbar u_1\wedge\ddbar u_1}{2},$$ we note that this is essentially an integral on $M$, since $u_1$ is zero on $|z|<r_\vep$. On $M\setminus B_1$, the integrand is of the order of $O(\vep^4)$, and hence the integral in this region is also of the order of $O(\vep^4)$. On $B_1$, $u_1$ behaves like $\vep^2\log|z|$, and $h(z) = O(|z|)$, and so $$\Big|h\frac{\ddbar u_1\wedge\ddbar u_1}{\omega^2}\Big|\leq C\vep^4|z|^{-3}.$$ So the integral on $B_1$ is bounded by $$\vep^4\int_{r_\vep}^1dr = O(\vep^4).$$
 The other integrals in the second and the third terms are even smaller. So we have the estimate $$\int_{\bM}\bl_{\Omega_1}(h)\frac{\Omega_1^2}{2} = \int_{\bM}h\frac{\oep^2}{2}  + O(\vep^{\lambda}),$$ for some $\lambda>2$

To compute the first integral, note that it is enough to integrate on $M_p$, since the exceptional divisor is of measure zero. Again expanding 
\begin{align*}
\int_{\bM}h\frac{\oep^2}{2} &= \int_{|z|>2r_\vep}h\frac{\omega^2}{2} + \int_{|z|<2r_\vep}h\frac{\omega_\vep^2}{2}\\
&= \int_M h\frac{\omega^2}{2} + \vep^2\int_{|z|<2r_\vep}h\Delta(\ga_2 \log|z|)\frac{\omega^2}{2} \\
&+ \vep^4\int_{|z|<2r_\vep} h\frac{\ddbar \ga_2\log|z|\wedge \ddbar \ga_2\log|z|}{2}\\
&=\int_M h\frac{\omega^2}{2} + O(\vep^4)
\end{align*} if $\al>2/3$. For instance, to estimate the second term, using the fact that $h = O(r_\vep)$, $\Delta(\ga_2\log|z|) = O(r_\vep^{-2})$ and $Vol(B_{2r_\vep}) = O(r_\vep^4)$, the integral is of the order of $O(\vep^2r_\vep^3)$, and hence smaller than $O(\vep^4)$.  This completes the proof of \eqref{integral estimate for h}. Applying this estimate to $h = \vep^{-2}(f_1 - \s_\omega)$, we see that $$\int_{\bM}\bl_{\Omega_1}(f_1-\s_\omega)\frac{\Omega_1^2}{2} = - 2\pi^2\vep^2 + \frac{\pi\vep^4}{2}\s_\omega + O(\vep^\kappa),$$ for some $\kappa>4$. Subtracting from \eqref{integral estimate for scal}, we obtain the required integral estimate.
\end{proof}

By Lemma \ref{lin2} and Lemma \ref{integral estimate}, there is a constant $C_\vep = O(\vep^{\kappa})$, for some $\kappa>4$ such that there is a solution $(u_2,h_2)$ to \begin{equation}\label{Second approx}
\sD_{\Omega_1}^*\sD_{\Omega_1}u + \l_{\Omega_1}(h) = \s(\Omega_1) - \l_{\Omega_1}(f_1) -C_\vep,
\end{equation}
such that $$||u_2||_{C^{4,\al}_{-1}} + |h_2| <C.$$ 

Denoting $f_2 = f_1 + h_2,$ and $\Omega_2 = \Omega_1 + \ddbar u_2$, we finally have the following estimate.

\begin{lem}\label{scal estimate second}
There exists a $\kappa>4$ such that $$||\s(\Omega_2) - \bl_{\Omega_2}(f_2)||_{C^{0,\al}_{-4}} = O(\vep^{\kappa}).$$
\end{lem}
The proof is identical to that of Lemma 30 in \cite{Sz15}, and so we skip it. 


\section{Proof of the main theorem}

We now complete the proof of the theorem using the contraction mapping principle. The proof runs almost identical to that in \cite{Sz15}, but since we have slightly worse bound on the inverse, we need to make some small modifications as in \cite{Sz12}. Let $\delta>0$ be small.  We want to find $(u,h)$ solving 
\begin{equation}\label{final equation}
\s(\Omega_2 + \ddbar u) = \bl_{\Omega_2+\ddbar u}(f_2+h).
\end{equation}

Or equivalently, we want to solve $$\s(\Omega_2) + \sL_{\Omega_2}(u) + Q_{\Omega_2}(u) = \bl_{\Omega_2}(f_2) + \bl_{\Omega_2}(h) + \mathrm{Re}\langle \nabla \bl(f_2),\nabla u\rangle + \mathrm{Re}\langle \nabla \bl(h),\nabla u\rangle.$$ Now let $\tilde \sF:C^{4,\al}_{-\delta}\times\overline{\mathfrak{h}}\rightarrow C^{0,\al}_{-\delta-4}$ be the linear operator 
\begin{align*}
\tilde \sF (u,h) &:=  \sL_{\Omega_2}(u) - \bl_{\Omega_2}(h) - \mathrm{Re}\langle \nabla \bl(f_2),\nabla u\rangle\\
&= -\sD_{\Omega_2}^*\sD_{\Omega_2}u  - \bl_{\Omega_2}(h) + \mathrm{Re}\langle\nabla_{\Omega_2}\s(\Omega_2),\nabla_{\Omega_2}u\rangle_{\Omega_2} - \mathrm{Re}\langle \nabla \bl(f_2),\nabla u\rangle.
\end{align*} By the remark following Proposition \ref{lin2}, and Lemma \ref{scal estimate second}, it follows that if $\vep<<1$, the operator $\tilde\sF$ is a small perturbation of the operator $\sF_1$. Then one can use the inverse $P_1$ to construct a right inverse $\tilde P$ to $\tilde \sF$ such that $$||\tilde P||<C\vep^{-\delta}.$$

Using this inverse, solving equation \eqref{final equation} is equivalent to solving the fixed point equation $$\sN(u,h) = (u,h),$$ where $$\sN(u,h) = \tilde P\Big( \bl_{\Omega_2}(f_2) -\s(\Omega_2)  - Q_{\Omega_2}(u) + \mathrm{Re}\langle \nabla \bl(h),\nabla u\rangle\Big).$$ 

\begin{lem}\label{contraction estimate}
There exists a constant $c_1>0$ such that for any $v_i\in C^{4,\al}$ and $g_i\in \overline{\mathfrak{h}}$, $i=1,2$,  satisfying $$||v_i||_{C^{4,\al}_2}+|g|<c_1\vep^{{\delta}},$$ we have the estimate $$||\sN(v_2,g_2) - \sN(v_1,g_1)||_{C^{4,\al}_{-\delta}} \leq \frac{1}{2}\Big(||v_2-v_1||_{C^{4,\al}_{-\delta}} + |g_2-g_1|\Big).$$
\end{lem}

\begin{proof}
By the estimate on $\tilde P$, it is enough to show that 
\begin{equation*}
||Q_{\Omega_2}(v_2) - \mathrm{Re}\langle \nabla \bl(g_2),\nabla v_2\rangle-Q_{\Omega_2}(v_1) + \mathrm{Re}\langle \nabla \bl(g_1),\nabla v_1\rangle||_{C^{0,\al}_{-\delta-4}} < \frac{\vep^{\delta}}{2}\Big(||v_2-v_1||_{C^{4,\al}_{-\delta}} + |g_2-g_1|\Big). 
\end{equation*}

By Lemma \ref{nonlinear terms estimate}, 
\begin{align*}
||Q_{\Omega_2}(v_2)-Q_{\Omega_2}(v_1)||_{C^{0,\al}_{-\delta-4}} &< C(||v_2||_{C^{4,\al}_2} + ||v_1||_{C^{4,\al}_2})||v_2-v_1||_{C^{4,\al}_{-\delta}} <Cc_1\vep^{\delta}||v_2-v_1||_{C^{4,\al}_{-\delta}}\\
&<\frac{\vep^{\delta}}{4}||v_2-v_1||_{C^{4,\al}_{-\delta}},
\end{align*}
if $c_1$ is chosen small enough. For the other term, the proof is identical to the one in \cite{Sz12}. 
\begin{align*}
||\mathrm{Re}\langle \nabla \bl(g_2),\nabla v_2\rangle &- \mathrm{Re}\langle \nabla \bl(g_1),\nabla v_1\rangle||_{C^{0,\al}_{-\delta-4}} \\
&\leq ||\nabla \bl(g_2)\cdot\nabla(v_2-v_1)||_{C^{0,\al}_{-\delta-4}} + ||\nabla v_1\cdot\nabla\bl(g_2-g_1)||_{C^{0,\al}_{-\delta-4}}\\
&\leq C\Big(|g_2|\cdot ||v_2-v_1||_{C^{4,\al}_{-\delta}} + ||v_1||_{C^{4,\al}_{-\delta}}|g_2-g_1|\Big)\\
&\leq Cc_1\vep^{\delta}\Big(||v_2-v_1||_{C^{4,\al}_{-\delta}} + |g_2-g_1|\Big)\\
&\leq \frac{\vep^\delta}{4}\Big(||v_2-v_1||_{C^{4,\al}_{-\delta}} + |g_2-g_1|\Big).
\end{align*} 
Note that we also used the fact that $||v||_{C^{4,\al}_{-\delta}}<||v||_{C^{4,\al}_2}$. Combining with the estimate for the $Q$ terms, we complete the proof of the Lemma.
\end{proof}

Now from Lemma \ref{scal estimate second}, it follows that $$||\sN(0,0)||_{C^{4,\al}_{-\delta}} <c_2\vep^{\kappa'},$$ for some constants $c_2>0$ and $\kappa'>4$, if we choose $\delta<\kappa-4$, and let $\kappa' = \kappa - \delta.$ With this $\kappa'$ and $c_2$, let $$\sU = \{(v,g)~|~ ||v||_{C^{4,\al}_{-\delta}}+|g|\leq 2c_2\vep^{\kappa'}\}.$$ Note that if $\vep$ is small enough, any $(v,g) \in \sU$ satisfies $$||v||_{C^{4,\al}_2}+|g|<c_1\vep^{\delta}.$$ For instance, the estimate on $v$ follows from the observation that $$|v||_{C^{4,\al}_2} <C\vep^{-2-\delta}||v||_{C^{4,\al}_{-\delta}} < \frac{c_1}{2}\vep^{\delta}$$ if $\vep$ is small enough, since $\kappa'-2-2\delta >0$. Then by Lemma \ref{contraction estimate}, for any $(v,g) \in \sU$, 
\begin{align*}
||\sN(v,g)||_{C^{4,\al}_{-\delta}} &< ||\sN(0,0)||_{C^{4,\al}_{-\delta}} + ||\sN(v,g) - \sN(0,0)||_{C^{4,\al}_{-\delta}}\\
&<c_2\vep^{\kappa'} + \frac{1}{2}\Big(||v||_{C^{4,\al}_{-\delta}} + |g|) \leq 2c_2\vep^{\kappa'}.
\end{align*}

This shows that $\sN$ maps $\sU$ into itself, and by Lemma \ref{contraction estimate} and the arguments above, is clearly a contraction on $\sU$. Since $\sU$ is clearly closed, by the contraction mapping principle there is a unique fixed point, which is the solution we seek. 

If $(u_3,h_3)$ is a solution to the fixed point equation, then  $u = u_1+u_2+u_3$ and $f = f_2+h_3$ solve equation \eqref{main thm equation}, and $(u_3,h_3)\in \sU$. In particular $|h_3| = O(\vep^\kappa)$, for some $\kappa>4$, and $$f =  \s_\omega -2\pi^2\vep^2(V^{-1}+\mu(p)) + \frac{\pi\s_\omega}{2}\vep^4(V^{-1}+\mu(p))+ O(\vep^\kappa).$$ Since the solution is found by the contraction mapping principle, we can carry out this construction at all $T$-invariant points and choose the constant in $O(\vep^\kappa)$ to be uniform.

\subsection*{Acknowledgement} I would like to thank G\'abor Sz\'ekelyhidi for introducing me to Conjecture \ref{Conj} and for sharing many useful insights. I would also like to thank Claudio Arezzo, Richard Bamler, Rafe Mazzeo, Michael Singer, Jian Song and Xiaowei Wang for stimulating discussions. 

\bibliographystyle{acm}

\bibliography{/Users/veddatar/Math/Bibliography/mybib}

\end{document}